\documentclass[11pt,english,reqno]{amsart}
\usepackage{amsmath, amssymb, amsbsy, amscd, amsthm, epsfig, amsfonts}
\usepackage{amsmath}
\usepackage{array}
\usepackage{hyperref}
\usepackage{algpseudocode}
\usepackage{amsfonts}
\usepackage[OT2,T1]{fontenc}
\usepackage{cancel}
\usepackage[small,nohug,heads=vee]{diagrams}
\diagramstyle[labelstyle=\scriptstyle]
\usepackage{latexsym}
\usepackage{graphicx}
\usepackage{setspace}
\usepackage[]{fontenc}
\usepackage{geometry}
\usepackage{fancyhdr}
\usepackage{enumerate}
\newarrow{Onto} ----{>>}
\makeatletter
\makeindex

 \theoremstyle{plain}
\newtheorem{theo}{Theorem}[section]
  \theoremstyle{remark}
  \newtheorem{rema}[theo]{Remark}
  \theoremstyle{plain}
  \newtheorem{corr}[theo]{Corollary}
  \theoremstyle{plain}
  \newtheorem{prop}[theo]{Proposition}
  \theoremstyle{plain}
  \newtheorem{lemm}[theo]{Lemma}
  \theoremstyle{definition}
  \newtheorem{defi}[theo]{Definition}
  \theoremstyle{definition}
  \newtheorem{exam}[theo]{Example}%

{\hspace{\stretch{1}}$\Box$\vspace{2ex}}
{\hspace{\stretch{1}}$\Box$\vspace{2ex}}

\usepackage{babel}
\makeatother

\def\ot{\otimes}

\DeclareMathOperator{\ord}{ord}
\DeclareMathOperator{\Supp}{Supp}

\DeclareSymbolFont{cyrletters}{OT2}{wncyr}{m}{n}
\DeclareMathSymbol{\Sha}{\mathalpha}{cyrletters}{"58}
\def\wt{\widetilde}

\DeclareMathOperator{\maxi}{max}
\DeclareMathOperator{\fini}{finite\ set}
\DeclareMathOperator{\primes}{prime}

\DeclareMathOperator{\degree}{deg}

\newfont{\msbm}{msbm10}
\newfont{\msbms}{msbm6}
\def\Q{{\mathbb Q}}
\def\Z{{\mathbb Z}}
\def\al{\alpha}

\def\ft{\wt f}
\def\Ft{\wt F}

\def\P{{\mathbb P}}
\def\G{{\mathcal G}}
\def\Ff{{\mathcal F}}
\def\ff{{h}}

\def\U{{\mathcal U}}
\def\k{{\Bbbk}}
\def\F{{\mathbb F}}
\def\I{{\mathcal I}}

\DeclareMathOperator{\Sel}{Sel}

\DeclareMathOperator{\Image}{Image}
\DeclareMathOperator{\Kernel}{Kernel}
\DeclareMathOperator{\Gal}{Gal}
\DeclareMathOperator{\Maps}{Map}
\DeclareMathOperator{\Genus}{Genus}
\def\sel{{\Sel^{(\mu)}(C,\Q)}}
\def\th{{\theta}}

\def\Th{{\Theta}}
\def\p{{\mathfrak p}}

\def\bi{\begin{itemize}}
\def\ol{\overline}
\def\ei{\end{itemize}}
\def\la{\label}

\def\O{{\mathcal O}}
\def\de{\delta}

\def\x{\times}

\newcommand{\be}{\begin{equation}}
\newcommand{\ee}{\end{equation}}
\newcommand{\ben}[1]{\begin{enumerate}[(#1)]}
\newcommand{\een}{\end{enumerate}}
\newcommand{\ba}{\begin{array}}
\newcommand{\ea}{\end{array}}
\newcommand{\bea}{\begin{eqnarray*}}
\newcommand{\eea}{\end{eqnarray*}}

\numberwithin{equation}{section}

\begin{document}

\title{Descent on superelliptic curves}

\thanks{The author is supported by the research grant FCT SFRH/BD/44011/2008.}
\author{Michael Mourao}

\email{M.Mourao@warwick.ac.uk}

\address{Mathematics Institute, University of Warwick}

\begin{abstract} 
We are concerned with the question of determining the set $C(\Q)$, where $C$ is a curve defined by an equation of the form $y^q=f(x)$, where $q$ is an odd prime and $f$ is a polynomial defined over $\Q$. This question can often be answered using a set which encapsulates information about local solubility of a particular collection of covers of $C$. We define this set and show how to compute it.
\end{abstract}
\maketitle

\section{Introduction}

When  working with the set of rational points $C(\Q)$ of an algebraic curve $C$, we often encounter examples where the  Hasse principle fails to decide whether this set is empty or not. The method of two-cover descent on hyperelliptic curves is based on the fact that, for a hyperelliptic curve $C$, there is a computable collection of covers $\phi_\al:D_\al\rightarrow C$, such that\be\la{initial}C(\Q)=\bigcup_{\al\in\fini}\phi_\al(D_\al(\Q)).\ee Therefore, when local-to-global arguments cannot be applied directly to $C$, the problem can be transferred to the one of looking for rational points on the covers. This is described explicitly in \cite{MR2521292}. In the present paper we extend this method to superelliptic curves $C$ defined by an equation of the form $y^q=f(x)$, where $q$ is an odd prime and $f\in\Q[x]$ is $q$-th power-free. The theory behind the process of performing descent on the Jacobian variety $J_C$ of such curves, defined over a field containing the relevant roots of unity, is studied in detail in \cite{MR1465369}. In \cite{brendan1}, an extension of the descent map to the Picard group of curves of this type is used to introduce new insights on the nature of $\Sha(J_C/\Q)$. Nevertheless, it is often much simpler and faster to avoid working with the rational points of the Jacobian and instead restrict to information obtained using only the initial curve. The Selmer set we define contains information which is sometimes sufficient to determine the set $C(\Q)$. We explain the necessary theory and present an explicit algorithm, similar to the one in \cite{MR2521292}, to compute this Selmer set. Note that we can not use a trivial extension of the existing routines because our algorithm is expected to deal with the possibility of singular points, which do not appear in the case of hyperelliptic curves.

The problem of finding points on superelliptic curves generalizes the problem of finding solutions to Thue equations (see \cite{MR1892843}) and can also be used to solve generalized Fermat equations. In Example \ref{exa3} we solve four such equations considered in \cite{MR1915212}, which were used by the authors as examples of the limitations of their approach. Thus in many situations descent arguments are more appropriate than other techniques.  As illustrated by Examples \ref{exa2} and \ref{exa3} it is often the case that $C$ is everywhere locally soluble, but its associated covers fail to be so, preventing $C$ from having any rational points, since the union in \eqref{initial} is comprised of empty sets. In Example \ref{exa4}, local information together with information obtained using subcovers, following a method proposed in \cite{MR2011330} and \cite{MR2156713}, is used to prove that the curve\[y^3=(x^2-3)(x^4-2)\]
\noindent has no rational points except from one rational point at $\infty$. By making descent applicable to singular superelliptic curves, we were able to prove in Theorem \ref{theo1} that the only pair $(a,b)\in\Z_{>0}^2$ satisfying
\[
b^3=\sum_{i=1}^ai^9
\]
\noindent is $(1,1)$.

\section{Descent on superelliptic curves}\la{se1}
\subsection{Preliminaries}\la{ss11}

\begin{defi}
Let $q$ be an odd prime and $n$ be a positive integer. We define a \textbf{superelliptic curve} to be a plane curve $C$ defined as the locus of the equation\[y^q=f(x)=a_nx^n+\ldots+a_1x+a_0,\]\noindent for some polynomial $f$ with coefficients in $\Q$.
\end{defi}

\begin{rema}
When $q\nmid n$ we think of $C$ as having a single point at infinity, otherwise $C$ has $q$ distinct points at infinity.
\end{rema}

\begin{prop}\label{changeofvars}
Every superelliptic curve is birational to a superelliptic curve satisfying an equation of the form $y^q=f(x)$, where $\text{deg}(f)=n$ and $q\mid n$.
\end{prop}

\begin{proof}
Suppose $C': (y')^q=f'(x')$ is a superelliptic curve and that $\text{deg}(f')=m=iq+j$ with $i\in\Z_{\geq 0}$ and $0\leq j <q$. Pick $\al\in\Q$ such that $f'(\al)\neq 0$. Define the polynomial $h$ by $f'(x')=h(x'-\al)$ and the polynomial $f$ of degree $n=q(i+1)$ by $f(x)=h(1/x)x^{q(i+1)}$. Let $C$ be the superelliptic curve defined by $y^q=f(x)$. Then $C'$ is birational to $C$ via the map $(x',y')\mapsto \left(\frac{1}{x'-\al},\frac {y'}{(x'-\al)^{i+1}}\right)$.
\end{proof}

\noindent In light of this proposition, we will assume that $a_n\neq 0$ and that $q\mid n$ throughout the rest of the paper. When this holds, we can also think of $C$ as being the locus of the equation\begin{equation}\la{curve}Y^q=F(X,Z)=a_nX^n+\ldots+a_1XZ^{n-1}+a_0Z^n\end{equation}in the weighted projective plane $\P^2(1,\frac{n}{q},1)$, where the variables $X$ and $Z$ have weight $1$ and $Y$ has weight $n/q$. We will be using the affine and projective descriptions of $C$ interchangeably. Without loss of generality, we may assume that $f$ is a $q$-th power-free polynomial with integer coefficients.

From now on let $\k$ denote one of $\Q$, $\bar\Q$, $\Q_p$, $\bar\Q_p$ (where $p$ can be any rational prime). Define $\Ff_\k$ to be the set of monic, irreducible over $\k$, polynomials such that 
\[
f=a_n\prod_{\ff\in\Ff_\k}\ff^{n_\ff},
\]
where $1\leq n_\ff\leq q-1$ for all $\ff\in\Ff_\k$. Denote the degree of $\ff$ by $d_\ff$. Let $A_\k$ be the semi-simple $\k$-algebra $\k[t]/(g(t))$, where 
\[
g=\prod_{\ff\in\Ff_\k}\ff.
\]
Note that $g$ is defined over $\Q$ and does not depend on $\k$. We denote its degree by $d$. $A_\k$ decomposes as a direct product of finite field extensions of $\k$ 
\[
A_\k=\prod_{\ff\in\Ff_\k}K_\ff=\prod_{\ff\in\Ff_\k}\k[t]/(\ff).
\]
Denote by $\th_\ff\in K_\ff$ the image of the generator $t$ under the quotient map and by $\Th_\k$ the set
\[
\Th_\k=\{\th_\ff\quad :\quad \ff\in\Ff_\k\}. 
\]

For the following definition we will assume that a point $(X,Y,Z)\in C(\k)$ is normalized such that if $\k=\Q$ then $X,Y,Z\in\Z$ with $\gcd(X,Z)=1$ and if $\k=\Q_p$ then $X,Y,Z\in\Z_p$ with either $Z=1$ or $X=1$, $Z\in p\Z_p$. We can always find such representations by scaling the points. 

\begin{defi}
For $\k=\Q$ and $\k=\Q_p$ define the component maps $\delta_{\ff}:C(\k)\rightarrow K_{\ff}^*/K_{\ff}^{*q}$
\[
\de_{\ff}(X,Y,Z)=\left\{\ba{ll} (X-\th_{\ff}Z)K_{\ff}^{*q} & \text{if } X-\th_{\ff}Z\neq 0, \\ \sqrt[n_{\ff}]{\Ft_{\ff}(X,Z)^{-1}K_{\ff}^{*q}} & \text{otherwise},\ea\right. \]where $\Ft_{\ff}$ is the two-variable polynomial with coefficients in $K_{\ff}$ defined by $(X-\th_{\ff}Z)^{n_{\ff}}\Ft_{\ff}(X,Z)=F(X,Z)$ and $\sqrt[n_{\ff}]{\Ft_{\ff}(X,Z)^{-1}K_{\ff}^{*q}}$ is defined to be the unique element $v\in K_{\ff}^*/K_{\ff}^{*q}$ such that $v^{n_{\ff}}=\Ft_{\ff}(X,Z)^{-1}$. 
\end{defi}

\begin{rema}
\begin{enumerate}[(1)]
\item Note that $v$ exists because the groups $K_{h}^*/K_{h}^{*q}$ have exponent $q$ and $\gcd(q,n_{h}) = 1$ and it is unique since these groups are $\F_q$-vector spaces. 
\item The component maps $\de_{h}$ are defined this way because for a point $(X,Y,Z)\in C(\k)$ with $F(X,Z)\neq 0$, we have that $(X-\th_{h}Z)K_{h}^{*q}\equiv\sqrt[n_{h}]{\Ft_{h}(X,Z)^{-1}K_{h}^{*q}}$, since $(X-\th_{h}Z)^{n_{h}}\Ft_{h}(X,Z)=F(X,Z)=Y^q$.
\end{enumerate}
\end{rema}

\begin{defi}
Let $\de_\k :C(\k)\rightarrow A_\k^*/A_\k^{*q}$ be 
\[
\de_\k=\left(\de_h\right)_{h\in\Ff_\k}.
\]
\end{defi}

\noindent In order to account for the fact that for any $\lambda\in \k^*$ and $(X,Y,Z)\in C(\k)$, $(X,Y,Z)=(\lambda X,\lambda^{n/q}Y,\lambda Z)\in C(\k)$, we quotient the codomain of $\de_\k$ by this action of scalars. So we define an action of $\k^*$ on $A_\k^*$ by
\[
\lambda \left(\al_h\right)_{h\in\Ff_\k} = \left(\lambda\al_h\right)_{h\in\Ff_\k},
\]
where $\lambda\in\k^*$ and  $\left(\al_h\right)_{h\in\Ff_\k}\in A_\k^*$. This action descends to an action of $\k^*/\k^{*q}$ on $A_\k^*/A_\k^{*q}$. We denote by $A_\k^*/\k^*A_\k^{*q}$ the quotient of $A_\k^*/A_\k^{*q}$ by this action.

\begin{defi}
Define the \textbf{\textit{descent map}} $\mu_\k$ to be the composition
\[
\begin{CD}C(\k) @>{\de_\k}>> A_\k^*/A_\k^{*q} @>{\pi_\k}>> A_\k^*/\k^*A_\k^{*q}\end{CD}
\]
where $\pi_\k$ is just the projection to the quotient.
\end{defi}

\subsection{The image of $\de_\Q$}\la{ss12}

The image of $\de_\Q$ is contained in a finite subgroup $A(q,\mathbf{S})$ of $A_\k^*/A_\k^{*q}$. To see this let us restrict our attention to finding the allowed possibilities for each of the $\#\Ff_\Q$ components. Let $h\in\Ff_\Q$ and set $\ft_{h} (x) = f(x)/(x-\th_{h})^{n_h} \in K_{h}[x]$.

Suppose $(X,Y,Z)\in C(\Q)$ with $X,Y,Z\in\Z$ and $X$ coprime with $Z$. We have that
\[
\de_{h}(X,Y,Z)=(X-\th_{h} Z)K_{h}^{*q}.
\]
Now suppose that $\p\nmid a_n\O_{K_{h}}$ is a prime ideal of the ring of integers $\O_{K_h}$ of the number field $K_h$. By assumption we have that
\[
\ord_\p(X-\th_{h} Z)\geq 0,\qquad\ord_\p(\Ft_{h}(X,Z))\geq 0.
\]
At this point we want to figure out which primes $\p$ appear in the factorization of $(X-\th_{h}Z)$, but not as a $q$-th power. So suppose $q\nmid \ord_\p(X-\th_{h} Z)$. Since $\gcd(n_{h},q)=1$, this implies that $q\nmid \ord_\p(\Ft_{h} (X,Z) )$. In particular we have $X\equiv \th_{h} Z$ and $\Ft_{h} (X,Z)\equiv 0$ modulo $\p$, which together give that $\Ft_{h} (\th_{h} Z,Z)=Z^{n-n_h}\ft_{h} (\th_{h})\equiv 0$ modulo $\p$. But if $Z\equiv 0$ modulo $\p$ then also $X\equiv 0$ modulo $\p$ which contradicts coprimality, so we have that $\p\in \Supp(\ft_{h} (\th_{h})\O_{K_{h}})$. By dropping the initial condition on $\p$ we have that if $q\nmid \ord_\p(X-\th_{h} Z)$ then $\p \in \Supp(a_n\ft_{h} (\th_{h})\O_{K_{h}})\subseteq \Supp(\Delta \O_{K_{h}})$, where $\Delta=a_n\text{Disc}(g)$. In other words \begin{equation*} (X-\th_{h} Z)\O_{K_{h}}=\p_1^{e_1}\ldots\p_l^{e_l}\I^q\end{equation*}where $\{\p_1,\ldots,\p_l\}=\Supp\left(a_n \ft_{h}(\th_{h})\O_{K_{h}}\right)$, $(e_1,\ldots,e_l)\in \F_q^l$ and $\I$ is a fractional ideal of $K_{h}$. 

Now define the sets of primes $S_h=\Supp(a_n\ft_{h}(\th_{h})\O_{K_{h}})$ for $h\in\Ff_\Q$. Then by the discussion above $\Image(\de_{h})\subseteq K_{h}(q,S_h)$ where
\[
K_{h}(q,S_h):=\{\al K_{h}^{*q} : q\mid\ord_\p(\al) \text{ for all } \p\notin S_h\}.
\]
Note that $K_{h}(q,S_h)$ is a finite subgroup of $K_{h}^*/K_{h}^{*q}$ (a proof of this can be found within the proof of Proposition VIII 1.6. in \cite{MR2514094}). Therefore we have that
\begin{equation}\label{AS}
\Image(\de_\Q)\subseteq\prod_{h\in\Ff_\Q}K_{h}(q,S_h)=:A(q,\mathbf{S})
\end{equation}
which is a finite subgroup of $A_\Q^*/A_\Q^{*q}$. 

\subsection{The image of $\de_\k$}\la{ss13}

\begin{defi}
Define the weighted norm homomorphism $N_{A/\k}:A_\k^*\rightarrow\k^*$ as
\[
N_{A/\k}\left((\al_h)_{h\in\Ff_\k}\right) = \prod_{h\in\Ff_\k}N_{K_{h}/\k}(\al_h)^{n_{h}}.
\]
Since $N_{A/\k}(A_\k^{*q})$ is a subgroup of $\k^{*q}$ we also get a homomorphism $\bar N_{A/\k}:A_\k^*/A_\k^{*q} \rightarrow \k^*/\k^{*q}$.
\end{defi}

\noindent By commutativity of the following diagram of norm homomorphisms
\[
\begin{CD}A_\k^* @>{N_{A/\k}}>> \k^* \\ @VVV  @VVV \\ A_\k^*/A_\k^{*q} @>{\bar N_{A/\k}}>> \k^*/\k^{*q}\end{CD}
\]
and the fact that 
\[
N_{A/\k}\left(\de_\k(X,Y,Z)\right)=\prod_{h\in\Ff_\k}N_{K_{h}/\k}(X-\th_{h} Z)^{n_{h}}=Z^n\x\prod_{h\in\Ff_\k} h(X/Z)^{n_{h}}=\frac{Y^q}{a_n}
\]
we can deduce that 
\[
\Image(\de_\k)\subseteq \bar N_{A/\k}^{-1}\left(\frac{1}{a_n}\k^{*q}\right)=:H_\k
\]
which, if non-empty, is a coset of the subgroup $\Kernel(\bar N_{A/\k})$ in $A_\k^*/A_\k^{*q}$. Combining this with the inclusion \eqref{AS} we deduce that 
\[
\Image(\de_\Q)\subseteq H_\Q\cap A(q,\mathit{\mathbf{S}})=:H_\Q(\mathit{\mathbf{S}}). 
\]

\subsection{The image of $\mu_\Q$}\label{ss14}

\begin{lemm}\la{lemm11}
Let $A_\Q$ be the semi-simple $\Q$-algebra associated to the curve $C$ and $A(q,\mathit{\mathbf{S}})$ be the subgroup of $A_\Q^*/A_\Q^{*q}$ defined in \eqref{AS}. Also set 
\[
T:=\{p\ \text{prime} :  \text{for all } h\in\Ff_\Q,\text{ and all }\p\in \Supp(p\O_{K_{h}}), \left(q\mid \ord_\p(p\O_{K_{h}})\right)\ \text{or}\ (\p\in S_h) \}.
\]
\[
\begin{diagram}\Q^*/\Q^{*q} & \rTo^\iota & A_\Q^*/A_\Q^{*q} & \rTo^{\pi_\Q} & A_\Q^*/\Q^*A_\Q^{*q} & \rTo & 1\\ \uInto & & \uInto & \ruTo & & & \\ \Q(q,T) & & A(q,\mathit{\mathbf{S}}) & & & & \end{diagram}
\]
Then 
\ben{i}
\item $\Image(\iota)\cap A(q,	\mathit{\mathbf{S}})=\iota\left(\Q(q,T)\right)$
\item $\pi_\Q\left(A(q,\mathit{\mathbf{S}})\right)\cong A(q,\mathit{\mathbf{S}})/\iota\left(\Q(q,T)\right).$
\een
\end{lemm}

\begin{proof}
\ben{i}
\item First suppose that $aA^{*q}\in \iota\left(\Q(q,T)\right)$. Then $q\mid \ord_p(a)$ for all $p\notin T$. Let $\p$ be a prime of $\O_{K_{h}}$ for some $h$, with $\p\notin S_h$. We know that 
\[
\ord_\p(a\O_{K_{h}})=\ord_\p(p\O_{K_{h}})\x \ord_p(a)
\]
so by the definition of $T$ at least one of the factors will be divisible by $q$ thus also their product, which implies that $aA^{*q}\in A(q,	\mathit{\mathbf{S}})$.\\For the opposite inclusion, suppose that $a\in \Q^*$ and $aA^{*q}\in A(q,\mathit{\mathbf{S}})$. Then $q\mid \ord_\p(a\O_{K_{h}})$ for all $h\in\Ff_\Q$ and for all $\p\notin S_h$ so if $p\notin T$, then $q$ must divide $\ord_p(a)$ since $q$ divides the product but not the first factor in the equality above.
\item We have
\[
\pi_\Q\left(A(q,\mathit{\mathbf{S}})\right)\cong \frac{A(q,\mathit{\mathbf{S}})}{\Kernel(\pi_\Q)\cap A(q,\mathit{\mathbf{S}})}=\frac{A(q,\mathit{\mathbf{S}})}{\Image(\iota)\cap A(q,\mathit{\mathbf{S}})}\overset{(i)}{=}\frac{A(q,\mathit{\mathbf{S}})}{\iota\left(\Q(q,T)\right)}
\]
\een
\end{proof}

Let us denote $\pi_\Q(H_\Q(\mathbf{S}))$ by $\bar H_\Q(\mathbf{S})$. Since $\Image(\de_\Q)\subseteq H_\Q(\mathbf{S})$, we have that \begin{equation}\la{imageofmufirst}
\Image(\mu_\Q)\subseteq\bar H_\Q(\mathbf{S}).
\end{equation}
In the following section we will see how $\Image(\mu_\Q)$ is contained in a potentially strict subset of $\bar H_\Q(\mathbf{S})$ and provide an algorithm to compute it.

\section{The Selmer set}\la{imagemup}

\subsection{Determining the image of $\mu_{\Q_p}$}\la{theimageofmup}

In this section we will provide an algorithm which determines $\Image(\mu_{\Q_p})$ for a rational prime $p$. The algorithm relies on the fact that points of $C$ which lie in a ``sufficiently small'' $p$-adic neighborhood, have the same image under $\mu_{\Q_p}$.

The diagram below is crucial in the process of refining the possible image of $\mu$ even further:
\[
\begin{CD}C(\Q)  @>{\mu_\Q}>>  A_\Q^*/\Q^*A_\Q^{*q}\\ @V{\iota_p}VV   @VV{r_p}V\\ C(\Q_p)  @>{\mu_{\Q_p}}>>  A_{\Q_p}^*/\Q_p^*A_{\Q_p}^{*q}\end{CD}
\]
By commutativity, if we have a rational point $P\in C(\Q)$ then $\mu_{\Q_p}\circ\iota_p(P)=r_p\circ\mu_\Q (P)$. Therefore we have that 
\begin{equation}\label{imageofmu}
\Image(\mu_\Q)\subseteq r_p^{-1}(\Image(\mu_{\Q_p}))\cap \bar H_\Q(\mathbf{S})\end{equation}

\begin{defi}\label{selmerset}
The \textbf{Selmer set} over $\Q$ of the superelliptic curve $C$, is defined as
\[
\sel=\left\{[\al]\in \bar H_\Q(\mathbf{S}) : r_p([\al])\in\Image(\mu_{\Q_p})\ \text{for all rational primes }p\right\}.
\]
\end{defi}
\begin{rema}
Strictly speaking, the set defined here corresponds to the \lq\lq fake \rq\rq\ Selmer set found generally in the literature (e.g. \cite{MR2521292} and \cite{brendan1}). Roughly, the difference between the fake and the actual Selmer set is that the latter distinguishes between covers (defined in Section \ref{coverssub}) $\phi_{\al}:D_{\al}\rightarrow C$ and $\phi_{\al'}:D_{\al'}\rightarrow C$ when $\phi_{\al}$ and $\phi_{\al'}$ are different even if $D_\al$ and $D_{\al'}$ are isomorphic. Since we are not using both sets, we omit the \lq\lq fake\rq\rq\ from the notation.
\end{rema}

\noindent After considering the inclusion \eqref{imageofmufirst} in the end of Section \ref{ss14} and the inclusions \eqref{imageofmu} for every rational prime $p$ we get that
\[
\Image(\mu_{\Q})\subseteq \sel.
\]

Let $h\in\Ff_{\Q_p}$ and denote by $\p_h$ the prime of $\O_{K_h}$. The following two lemmas are used to show that the analytic space $C(\Q_p)$ can be covered by a finite number of neighborhoods, where the map $\mu_{\Q_p}$ is constant.  In practice $X_k$ (or $X$) will be a finite precision approximation to the first coordinate of a point $(X',Y',1)\in C(\Q_p)$.

\begin{lemm}\la{closeenough}
Suppose that $X',X_k\in\Z_p$ with $\ord_p(X'-X_k)\geq k$
\ben{i}
\item If $k\geq\frac{2\ord_{\p_h}(q)+\ord_{\p_h}(X_k-\th_h)+1}{e_{{\p_h}/p}}$ then $(X_k-\th_h)K_{h}^{*q}=(X'-\th_h)K_{h}^{*q}$.
\item If $k\geq\frac{2\ord_{\p_h}(q)+\ord_{\p_h}(\ft_h (X_k))+1}{e_{{\p_h}/p}}$ then $\ft_h (X_k)K_{h}^{*q}=\ft_h (X')K_{h}^{*q}$.
\een
\end{lemm}

\begin{proof}
\ben{i}
\item By the assumption $X'=X_k+up^k$ where $u\in\Z_p$. So
\[
\frac{X'-\th_h}{X_k-\th_h}=1+\frac{up^k}{X_k-\th_h}.
\]
Now let $\tau(t)=\frac{X'-\th_h}{X_k-\th_h}-t^q$. By Hensel's lemma we have that the following is a sufficient condition for $\tau$ to have a solution in $K_{h}$,
\[
\ord_{\p_h}(\tau(1))\geq 2\ord_{\p_h}\left(\frac{d\tau}{dt}(1)\right)+1\]\[k\ord_{\p_h}(p)-\ord_{\p_h}(X_k-\th_h)\geq 2\ord_{\p_h}(q)+1\]So as long as the condition of the lemma is satisfied Hensel's lemma ensures that $(X_k-\th_h)$ and $(X'-\th_h)$ are the same modulo $K_{h}^{*q}$.
\item This is very similar to the previous part. Just use the fact that\\ $\ft_h (X')=\ft_h (X_k+up^k)=\ft_h (X_k)+vp^k$ where $v\in\Z_p$ and set $\tau(t)=\frac{\ft_h (X')}{\ft_h (X_k)}-t^q$.
\een
\end{proof}

\begin{lemm}\label{lemterminate}
If $\{X_k\}_{k=1}^\infty\subset\Z_p$ is a sequence satisfying $\ord_p(X'-X_k)\geq k$  for some $X'\in\Z_p$ and every $k$, then there exists $N\in\Z_{>0}$ such that $X_N$ satisfies at least one of conditions (i) or (ii) of Lemma \ref{closeenough}.
\end{lemm}

\begin{proof}
Suppose such $N$ does not exist. This means that for every $k$ we have
\[
\min\left\{\frac{2\ord_{\p_h}(q)+\ord_{\p_h}(X_k-\th_h)+1}{e_{{\p_h}/p}},\frac{2\ord_{\p_h}(q)+\ord_{\p_h}(\ft_h (X_k))+1}{e_{{\p_h}/p}}\right\}>k
\]
and therefore both $\ord_{\p_h}(X_k-\th_h)$ and $\ord_{\p_h}(\ft_h(X_k))$ tend to infinity as $k$ tends to infinity. But since $\{X_k\}_{k=1}^\infty$ converges to $X'$ we have that $(X'-\th_h)=\ft_h(X')=0$, a contradiction.
\end{proof}

A clear distinction between the case of hyperelliptic ($q=2$) and superelliptic ($q>2$) curves is that, a superelliptic curve is allowed to have singularities, since $f$ being $q$-th power-free is no longer equivalent to $f$ not having repeated roots. At this point we would like to use some version of Hensel's Lemma to determine whether our finite precision $X_k$ lifts to $\Z_p$ as the first coordinate of a point $(X',Y',1)\in C(\Q_p)$. We have to be careful not to ask this question for points approximating one of the singularities as that would result in an infinite loop. Thus we have to determine the size of the $\mu_{\Q_p}$-constant neighborhood around each singularity (which is defined over $\Q_p$) in advance and compute its image.

Let
\[
\Ff_{\Q_p}^{lsi}=\left\{h \in \Ff_{\Q_p} : \degree(h)=1,n_h>1,\th_h\in\Z_p\right\},
\]
where the exponent $lsi$ stands for linear,singular and integral. Elements of this set correspond to the singular points on $C$ that are defined over $\Q_p$, but are of the form $(\th_h,0,1)$ with $\th_h\in\Z_p$. The last condition arises because we split the computation into two parts, the first being the determination of the image under $\mu_{\Q_p}$ of points of the form $(X',Y',1)\in C(\Q_p)$, with $X',Y'\in\Z_p$.

Consider the following functions:

\begin{algorithmic}[1]
\Function{SizeOfNeighborhood}{$h$}
	\State $k_h\gets 0$
	\State \textsc{Finish} $\gets$ \textbf{false}
	\While{$\left[\right.$ \textsc{Finish} = \textbf{false}$\left.\right]$}
		\State $k_h\gets k_h+1$	
		\State $X\gets X\in \Z\subset\Z_p : \ord_{\p_h}(X-\th_h)\geq k_h$
		\State \textsc{ValList}$\gets\left\{\frac{2\ord_{\p_{h'}}(q)+\ord_{\p_{h'}}(X-\th_{h'})+1}{e_{{\p_{h'}}/p}}:h'\in\Ff_{\Q_p}\setminus\{h\}\right\}\cup \left\{2\ord_{\p_h}(q)+\ord_{\p_h}(\ft_h (X))+1\right\}$
		\If{$\left[\right.$ $\maxi(\textsc{ValList})\leq k_h$ $\left.\right]$}
			\State \textsc{Finish} $\gets$ \textbf{true}
		\EndIf
	\EndWhile\\
	\hskip 0.6cm \Return{$k_h,X$}
\EndFunction
\end{algorithmic}

Note that the function \textsc{SizeOfNeighborhood} only makes sense when $\degree(h)=1$ otherwise we would not be able to find an $X$ satisfying the condition of step 6 for every given $k_h$ (that would imply that $\th_h\in\Z_p$), and we will actually only apply it to elements of $\Ff_{\Q_p}^{lsi}$. This function has a double use: The returned value of $X$ will be used to compute $\mu_{\Q_p}(\th_h,0,1)$, and $k_h$ will keep track of the size of the $\mu_{\Q_p}$-constant neighborhood around the singularity. Now for $h\in\Ff_{\Q_p}^{lsi}$ set 
\[
\U_h=\left\{X\in \Z_p : \ord_p(X-\th_h)\geq k_h\right\}
\]
and
\[
\U=\bigcup_{h\in\Ff_{\Q_p}^{lsi}}\U_h.
\]

The following function can be thought of as partitioning $\Z_p$ into neighborhoods, with the partition becoming finer close to the singularities. Then the function \textsc{LocalImage} will test each of these neighborhoods for elements that lift to points on $C$, and if necessary partition them further into $\mu_{\Q_p}$-constant parts.
\begin{algorithmic}[1]
\Function{ComputeInputList}{$\textsc{Reps},k,\textsc{List}$}
	\For{$X\in\textsc{Reps}$}
		\If{$\left[\right.$ $\exists h\in\Ff_{\Q_p}^{lsi}:\left[\right. k_h>k$  \textbf{and}  $\ord_p(X-\th_h)\geq k_h$ $\left.\right]$ $\left.\right]$}
		\hskip 1.7cm \State \textsc{List} $\gets$\textsc{ComputeInputList}$\left(\left\{X+tp^k : t \in\{0,\ldots ,p-1\}\right\},k+1,\textsc{List}\right)$
		\Else
			\State \textsc{List} $\gets$ \textsc{List} $\cup \{(X,k)\}$
		\EndIf
	\EndFor\\
	\hskip 0.6cm \Return{\textsc{List}}
\EndFunction
\end{algorithmic}

At this point we should stress that the functions \textsc{SizeOfNeighborhood} and \textsc{ComputeInputList} would be redundant if there were no singularities, and the function \textsc{LocalImage} would be sufficient to compute the local image. In that case the input (\textsc{List}=$\{0,\ldots,p-1\}$,\textsc{Image}=$\{ \}$) would produce the require result.

\begin{algorithmic}[1]
\Function{LocalImage}{$\textsc{List},\textsc{Image}$} 
	\For{$\left[\right.$ $(X,k)\in \textsc{List}$ $\left.\right]$} 
		\If{$\left[\right.$ $X\in\mathcal{U}$ \textbf{or} $\nexists (X',Y',1)\in C(\Q_p):\ord_p(X'-X)\geq k$ $\left.\right]$} 
			\hskip 1.7cm \State \textsc{List} $\gets$ \textsc{List}$\setminus\{(X,k)\}$
		\Else
			\State \textsc{ValList}$\gets \left\{\left[\frac{2\ord_{\p_h}(q)+\ord_{\p_h}(X-\th_h)+1}{e_{{\p_h}/p}},\frac{2\ord_{\p_h}(q)+\ord_{\p_h}(\ft_h (X))+1}	{e_{{\p_h}/p}}\right]: h\in\Ff_{\Q_p}\right\}$
			\If {$\left[\right.$ $\forall [k_1,k_2]\in$\textsc{ValList} $\min(k_1,k_2)\leq k$ $\left.\right]$}
				\State \textsc{NewElement} $\gets 1$
				\For{$\left[\right.$ $[k_1,k_2]\in$\textsc{ValList} $\left.\right]$}
					\If {$\left[\right.$ $k_1\leq k_2$ $\left.\right]$}
						\State \textsc{NewElement} $\gets$ \textsc{NewElement}$\x (X-\th_h)K_{h}^{*q}$
					\Else
						\State \textsc{NewElement}$\gets$ \textsc{NewElement} $\x \sqrt[n_h]{\ft_h(X)^{-1}K_{h}^{*q}}$
					\EndIf
				\EndFor
				\State \textsc{NewElement}$\gets\ \pi_{\Q_p}\left(\textsc{NewElement}\right)$
				\hskip 2.3cm \State \textsc{Image} $\gets$ \textsc{Image} $\cup \{$\textsc{NewElement}$\}$
			\Else\\
				\hskip 2.3cm \textsc{NewList}$\gets\{(X+tp^k,k+1) : t \in\{0,\ldots ,p-1\}\}$\\
				\hskip 2.3cm \textsc{Image}$\gets$\textsc{LocalImage}($\textsc{NewList},\textsc{Image}$)
			\EndIf
		\EndIf
	\EndFor\\
	\hskip 0.6cm \Return{\textsc{Image}}
\EndFunction
\end{algorithmic}

With the help of the \textsc{SizeOfNeighborhood} function, we pre-compute and store in the variable \textsl{image} the images of the singular points under $\mu_{\Q_p}$. Using \textsc{LocalImage} with initial input\[\left(\textsc{ComputeInputList}\left(\{0,\ldots,p-1\},1,\{\}\right),\text{image}\right)\]we get as output a set $V_1\subseteq A_{\Q_p}^*/\Q_p^*A_{\Q_p}^{*q}$ which satisfies $\mu_{\Q_p}(U_1)=V_1$, where $U_1=\{(X,Y,1)\in C(\Q_p):X,Y\in\Z_p\}$. With slight modifications to the routine above we can also obtain as output a set $V_2$ such that $\mu_{\Q_p}(U_2)=V_2$, where $U_2=\{(1,Y,pZ)\in C(\Q_p):Y,Z\in\Z_p\}$. Thus we obtain the complete image of $\mu_{\Q_p}$, since $U_1\cup U_2=C(\Q_p)$.

\begin{rema}
We can be certain that the routine LocalImage terminates after a finite number of steps because of Lemma \ref{lemterminate}. An infinite loop would correspond to a sequence $\{X_k\}_{k=1}^\infty$ converging to some $X'\in\Z_p$ satisfying $X'-\th_h=\ft_h (X')=0$ which is impossible. Also note that the If statement at step 3, ensures that if $X\in\mathcal{U}$, in other words if our approximate value is very close to the first coordinate of one of the singular points, then it is excluded from \textsc{List} and we do not try to lift it using Hensel's Lemma, which would have resulted in an infinite loop.
\end{rema}

\subsection{The corresponding covers}\label{coverssub}

For every $\al\in A_\k^*$ such that $[\al]:=\al\Q^*A^{*q}\in \bar H_\k$ we can construct an unramified cover of $C$ of degree $q^{d-2}$
\[
\phi_\al : D_\al\rightarrow C,
\]
defined over $\k$ satisfying the properties
\bea 
D_\al(\k)\neq\emptyset & \Leftrightarrow & \left[\al\right]\in \Image(\mu_\k)\\
\left[\al\right] = \left[\al'\right] & \Rightarrow & D_\al\cong D_{\al '}  
\eea

First let us give an equivalent description of the $\k$-algebra $A_\k$. Denote the absolute Galois group $\Gal(\bar \k/\k)$ by $\G_\k$. We fix embeddings $K_h\hookrightarrow\bar\k$ for all $h\in\Ff_\k$ that are compatible in the sense that they agree on the intersections $K_h\cap K_{h'}$ for $h,h'\in\Ff_{\k}$. We then get an inclusion $\Theta_\k\hookrightarrow\Theta_{\bar\k}$ and we can treat elements $\al_h\in K_h$ as elements of $\bar\k$ and elements of $\Th_\k$ as elements of $\Th_{\bar\k}$.

\begin{lemm}
\[
A_\k\cong \Maps_\k(\Th_{\bar\k},\bar \k),
\]
where the right hand side is the set of all $\G_\k$-equivariant maps from $\Th_{\bar\k}$ to $\bar\k$.
\end{lemm}

\begin{proof}
The isomorphism is given by
\[
(\al_h)_{h\in\Ff_\k}\mapsto\left(\th_{h'}\mapsto\ ^\sigma\al_h : \sigma\in\G_\k,\ h\in\Ff_\k,\ ^\sigma\th_h=\th_{h'}\right)
\]
with inverse
\[
\xi\mapsto \left(\xi(\th_h)\right)_{\th_h\in\Th_\k\subseteq\Th_{\bar\k}}.
\]
\end{proof}

Let $\al\in A_\k^*$ such that $[\al]\in\bar H_\k$. Since $\al A^{*q}\in H_\k$, there exists $v\in\k^*$ with $a_nN_{A/\k}(\al)=v^q$. Let $D_\al$  be the variety in $\P^{d-1}\times C$ defined by
\begin{align}
\label{coverseq}\left((u_h)_{h\in\Ff_{\bar\k}},(X,Y,Z)\right)\in D_\al  \Leftrightarrow & \exists \lambda\neq 0 \text{ s.t. }\lambda\al (\th_h)u_h^q= X-\th_{h}Z\\
\nonumber & \text{for all } h\in\Ff_{\bar\k}\text{ and}\\
\nonumber & \lambda^{n/q}v\prod_{h\in\Ff_{\bar\k}}u_h^{n_h}=Y
\end{align}

\noindent We use a description of the covers $D_\al$ similar, at least in terms of their ambient space, to the one found in \cite{brendan1} and \cite{Siksek01022012}. We equip the first factor, $\P^{d-1}$, with the twisted $\G_\k$-action which permutes coordinates in the same way it permutes $\Th_{\bar\k}$. In other words if $\sigma\in\G_\k$ satisfies $^\sigma\th_h=\th_{h'}$, then $^\sigma u_{h}=u_{h'}$. It is then obvious from the definition that $D_\al$ is actually defined over $\k$. Projection to the second factor gives rise to the required covering map
\[
\phi_\al:D_\al\rightarrow C.
\]

\begin{lemm}
The map $\phi_\al: D_\al\rightarrow C$ is unramified of degree $q^{d-2}$. In particular $D_\al$ is a curve.
\end{lemm}

\begin{proof}
Let $(X,Y,Z)\in C(\bar \k)$ and suppose $X-\th_{h_0}Z\neq 0$ for some $h_0\in\Ff_{\bar\k}$. Then $u_{h_0}\neq 0$ so we can set $u_{h_0}=1$. By doing this we also fixed $\lambda=\frac{X-\th_{h_0}Z}{\al(\th_{h_0})}$. So for all $u_{h}$ with $h\neq h_0$ we have
\[
u_{h}^q=\frac{\al(\th_{h_0})(X-\th_{h}Z)}{\al(\th_{h})(X-\th_{h_0}Z)}.
\]
Since we are over $\bar\k$, if there does not exist $h\in\Ff_{\bar\k}$ such that $X-\th_{h}Z=0$, then there are exactly $q$ different choices for the value of each of the $d-1$ $u_{h}$\rq s. Once $d-2$ of them have been chosen, the remaining one is decided by the relation
\[
\lambda^{n/q}v\prod_{h\in\Ff_{\bar\k}} u_h^{n_h} =Y.
\]
The fact that there is a unique choice for the value of the remaining coordinate uses that $\gcd(n_h,q)=1$ for every $h\in\Ff_{\bar\k}$. On the other hand if $X-\th_{h'}Z=0$ for some $h'\in\Ff_{\bar\k}\setminus\{h_0\}$ then $u_{h'}=0$ and there are $q$ choices for the remaining $d-2$ $u_{h}$\rq s. The extra relation in this case does not decide the value for any of them. We see that in both cases the fiber of $\phi_\al$ over $P\in C(\bar\k)$ contains exactly $q^{d-2}$ points.
\end{proof}

\begin{prop}\la{coversprop}
$D_\al(\k)\neq\emptyset$ if and only if $[\al]\in\Image(\mu_\k)$. Furthermore, if $[\al]=[\al']$ then $D_\al\cong D_{\al'}$ over $\k$. In other words, up to $\k$-isomorphism, $D_\al$ only depends on the class $[\al]$ in $A_\k^*/\k^*A_\k^{*q}$.
\end{prop}

\begin{proof}
Let $P=\left((u_{h})_{h\in\Ff_{\bar\k}},(X,Y,Z)\right)\in D_\al(\k)$. Projecting gives $\phi_\al(P)=(X,Y,Z)\in C(\k)$ and if $X-\th_hZ\neq 0\ \forall\ h\in\Ff_{\bar\k}$ then 
\[
\delta_\k\left(\phi_\al(P)\right)=\left((X-\th_hZ)K_h^{*q}\right)_{\th_h\in\Th_\k\subseteq\Th_{\bar\k}}=\left(\lambda \alpha_hu_h^qK_h^{*q}\right)_{\th_h\in\Th_\k\subseteq\Th_{\bar\k}}
\] 
therefore $\mu_\k(\phi_\al(P))=[\al]$. On the other hand if there exists $h'\in\Ff_{\bar\k}$ such that $X-\th_{h'}Z=0$ then by definition of $\delta_{h'}$ we have
\begin{align*}
\delta_{h'}(\phi_\al(P)) & =  \sqrt[n_{h'}]{\Ft_{h'}(X,Z)^{-1}K_{h'}^{*q}}\\
 & = \sqrt[n_{h'}]{\frac{1}{a_n\displaystyle\prod_{h\in\Ff_{\bar\k}\setminus{h'}}(X-\th_{h}Z)^{n_h}}K_{h'}^{*q}}\\
 & = \sqrt[n_{h'}]{\frac{1}{a_n\lambda^{n-n_{h'}}\displaystyle\prod_{h\in\Ff_{\bar\k}\setminus{h'}}\left(\al(\th_{h})u_{h}^q\right)^{n_h}}K_{h'}^{*q}}\\
 & = \sqrt[n_{h'}]{\frac{(\lambda \al_{h'})^{n_{h'}}}{a_n\lambda^nN_{A/\k}(\al)}K_{h'}^{*q}}\\
 & = \sqrt[n_{h'}]{\frac{(\lambda \al_{h'})^{n_{h'}}}{\left(\lambda^{n/q}v\right)^q}K_{h'}^{*q}}\\
 & = \sqrt[n_{h'}]{(\lambda \al_{h'})^{n_{h'}}K_{h'}^{*q}}\\
 & = \lambda \al_{h'} K_{h'}^{*q}.
\end{align*}
So again $\mu_\k(\phi_\al(P))=[\al]$, since all the other components of $\delta_\k$ are evaluated without using cofactors.

For the other implication, suppose $[\al]\in\Image(\mu_\k)$. This means there exist $X,Y,Z\in \O_\k$, $\lambda\in\k^*$ and $\beta\in A_\k^*$ such that\[\lambda\alpha(\th_{h})\beta(\th_{h})^q=X-\th_{h} Z\] for $\th_h\in\Th_{\k}$. After conjugating these relations by elements of $\G_\k$ and using the fact that $\al$ and $\beta$ are $\G_\k$-equivariant, we obtain the corresponding relations for $\th_h\in\Th_{\bar\k}\setminus\Th_\k$. Then $\left((\beta(\th_{h}))_{h\in\Ff_{\bar\k}},(X,Y,Z)\right)\in D_\al(\k)$. Note that if for some $h'\in\Ff_{\k}$ we have $X-\th_{h'}Z=0$ then $d_{h'}=1$ and the corresponding coordinate $u_{h'}$ is equal to zero.

For the last statement suppose that we have $\al ,\al'\in A_\k^*$ with $[\al]=[\al']$. This implies that there exist $\lambda\in\k^*$ and $\beta\in A_\k^*$ such that $\al=\lambda\al'\beta^q$. By definition of the covers it is not hard to see that we can then map $D_\al$ to $D_{\al'}$ via $\left((u_{h})_{h\in\Ff_{\bar\k}},(X,Y,Z)\right)\mapsto \left((\beta(\th_{h})u_{h})_{h\in\Ff_{\bar\k}},(X,Y,Z)\right)$ which is clearly an isomorphism and it is defined over $\k$ since it is invariant under the twisted $\G_\k$-action.
\end{proof}

\begin{corr}\label{cor1}
Let $\mathcal{H}$ be any subset of $\bar H_\Q(\mathbf{S})$ containing $\Image(\mu_\Q)$, then 
\[
C(\Q)= \bigcup_{[\al]\in\mathcal{H}}\phi_\al\left(D_\al(\Q)\right).
\]
In particular the above equality holds for $\mathcal{H}=\sel$.
\end{corr}

\begin{proof}
Since for all $\al\in A_\Q^*$ such that $[\al]\in \bar H_\Q$, $D_\al$ and $\phi_\al$ are defined over $\Q$, we know that the right hand side is contained in $C(\Q)$. Also, from the proof of Proposition \ref{coversprop} we deduce that if we have $P\in C(\Q)$ and $\mu_\Q(P)=[\al]$ then there exists $Q\in D_{\al}(\Q)$ with $\phi_{\al}(Q)=P$. On the other hand, if $[\al]\in\mathcal{H}\setminus\Image(\mu_\Q)$ then $D_\al(\Q)=\emptyset$. 
\end{proof}

\begin{prop}
The curves $D_\al$ are non-singular.
\end{prop}

\begin{proof}
To show this let us restrict to the affine patch where $u_{h_0}\neq 0$ and $Z\neq 0$ for some $h_0\in\Ff_{\bar\k}$. We thus assume that $u_{h_0}=Z=1$, label the elements of $\Ff_{\bar\k}\setminus\{h_0\}$ using an index $i\in\{1,\ldots,d-1\}$ and rename $u_i:=u_{h_i},\ \th_i:=\th_{h_i}$ and $n_i:=n_{h_i}$ to simplify the notation. The defining equations for $D_\al$ in this patch become
\begin{align*}
\al(\th_{i})(X-\th_{0})u_{i}^q =& \al(\th_{0})(X-\th_{i})\qquad\text{for } i\in\{1,\ldots,d-1\}\\
v(X-\th_{0})^{n/q}\prod_{i=1}^{d-1}u_{i}^{n_{i}}=& Y. 
\end{align*}
We get the following $(d+1)\times d$ matrix of partial derivatives which represents a linear map whose cokernel is the cotangent space of $D_\al$ at a generic point of the affine patch.
\be\label{matrix}\small
\left(
\begin{array}{@{}c@{}c@{}c@{}cc@{}}
\al(\th_{1})u_{1}^q - \al(\th_{0}) & \al(\th_{2})u_{2}^q - \al(\th_{0}) & \ldots & \al(\th_{{d-1}})u_{{d-1}}^q - \al(\th_{0}) & v\left(\frac{n}{q}\right)(X-\th_{0})^{n/q-1}\displaystyle\prod_{i=1}^{d-1}u_{i}^{n_{i}}\\
q\al(\th_{1})(X-\th_{0})u_{1}^{q-1} & 0 & \ldots & 0 & vn_{1}(X-\th_{0})^{n/q}u_{1}^{n_{1}-1}\displaystyle\prod_{i\neq 0,1}u_{i}^{n_{i}}\\
0 & q\al(\th_{2})(X-\th_{0})u_{2}^{q-1} & \ldots & 0 & vn_{2}(X-\th_{0})^{n/q}u_{2}^{n_{2}-1}\displaystyle\prod_{i\neq 0,2}u_{i}^{n_{i}}\\
\vdots &\vdots &\vdots &\vdots &\vdots \\
0 & 0 &\ldots & q\al(\th_{d-1})(X-\th_{0})u_{d-1}^{q-1} & vn_{d-1}(X-\th_{0})^{n/q}u_{d-1}^{n_{d-1}-1}\displaystyle\prod_{i\neq 0,d-1}u_{i}^{n_{i}}\\
0 & 0 & \ldots & 0 & -1
\end{array}
\right)
\ee

\noindent This matrix has rank $d$ at every point of the affine patch. To see this note that the first $d-1$ entries of the first row, can never be zero, since this would contradict the fact that the roots of $g$ are distinct. For the same reason, at most one row can be identically zero and when this happens the $d\times d$ matrix obtained by deleting that row has rank $d$. A similar argument for all of the $2d$ affine patches covering $D_\al$ shows that $D_\al$ is non-singular. 
\end{proof}

\begin{prop}\label{nonsingfinite}
If $p$ is a rational prime such that $p\nmid q\Delta$ and $\al\in A_\Q^*$ such that $[\al]\in \bar H_\Q(\mathbf{S})$, then $D_{\beta}$ has good reduction, where $\beta\in A_{\Q_p}^*$ is such that $[\beta]=r_p([\al])$.
\end{prop}

\begin{proof}
Since $p$ does not divide $q$ or the discriminant of $g$, $\al\in A_\Q^*$ can be chosen such that $\ord_\p\left(\al(\th_h)\right)=0$ for every $h\in \Ff_{\bar\Q}$ and every $\p\mid p$. Take $\beta$ to be the image of this $\al$ under the inclusion $\ot \Q_p: A_\Q^*\rightarrow A_{\Q_p}^*$. Then $[\beta]=r_p\left([\al]\right)$ and also $\ord_\p\left(\beta(\th_h)\right)=0$ for every $h\in \Ff_{\bar \Q_p}$ and every $\p\mid p$. Thus the defining equation
\[\lambda\beta(\th_h)u_h^q=X-\th_hZ
\]
of $D_\beta$ can be reduced to an equation modulo $\p_h$ for every $h\in\Ff_{\bar\Q_p}$, where $\p_h$ is the prime of $K_h$ above $p$. We thus get an unramified cover $\ol{\phi_\beta}:\ol{D_\beta}\rightarrow \ol{C}$ defined over $\F_p$. Furthermore we know that $\ol{D_\beta}$ is non-singular since the reduction of the matrix of partial derivatives in \eqref{matrix} has full rank.
\end{proof}

\begin{prop}\label{genus}
The genus $G$ of the covers $D_\al$ is equal to $q^{d-2}\left(\frac{d(q-1)}{2}-q\right)+1$.
\end{prop}

\begin{proof}
Let $\psi:C\rightarrow \P^1$ be the map $(X,Y,Z)\mapsto (X,Z)$. This has degree $q$ and is ramified above the $d$ points $(X,Z)$ satisfying $F(X,Z)=0$. Since $\phi_\al$ is unramified of degree $q^{d-2}$, the composition $\psi\circ\phi_\al:D_\al\rightarrow\P^1$ has degree $q^{d-1}$ and is ramified at the same points as $\psi$. Each ramification point has exactly $q^{d-2}$ preimages, each of them with ramification index $q$. The Riemann-Hurwitz formula applied to $\psi\circ\phi_\al$ yields
\begin{align*}
2G-2= & (2\Genus(\P^1)-2)\degree(\psi\circ\phi_\al)+\sum_{P\in \P^1,Q\in(\psi\circ\phi_\al)^{-1}(P)}(e_Q-1)\\
2G-2= & -2q^{d-1}+dq^{d-2}(q-1)\\
G= & q^{d-2}\left(\frac{d(q-1)}{2}-q\right)+1
\end{align*}
\end{proof}

\begin{defi}\label{useful}
Define the set of \textbf{\textit{useful primes}} to be
\[
S_{up}:=\left\{p\text{ rational prime }: p\mid q\Delta\quad\text{or}\quad\sqrt{p}+\frac{1}{\sqrt{p}}\leq 2G\right\}.
\]
\end{defi}

\begin{prop}
Suppose that $p$ is a rational prime and $p\notin S_{up}$. Then $\bar H_\Q(\mathbf{S})\subseteq r_p^{-1}\left(\Image(\mu_{\Q_p})\right)$.
\end{prop}

\begin{proof}
Let $\al\in A_\Q^*$ such that $[\al]\in\bar H_\Q(\mathbf{S})$ and $\beta\in A_{\Q_p}^*$ such that $[\beta]=r_p\left([\al]\right)$. By Propositions \ref{nonsingfinite} and \ref{genus} we know that $\ol{D_\beta}$ is a non-singular curve of genus $G$ over $\F_p$. Then by the Hasse-Weil inequality we have that $\# \ol{D_{\beta}}(\F_p)>0$ and we can use Hensel's Lemma to lift to a point in $D_{\beta}(\Q_p)$. By Proposition \ref{coversprop} this is equivalent to $[\beta]=r_p([\al])\in \Image(\mu_{\Q_p})$.
\end{proof}

\begin{corr}\la{finite}
\[
\sel=\{\al\in H_\Q : r_p(\al)\in\Image(\mu_{\Q_p})\text{ for all }p\in S_{up}\}
\]
\end{corr}

\subsection{Computational efficiency}

The groups $A(q,\mathit{\mathbf{S}})$, $\Q(q,T)$ and the homomorphism $\iota$ defined in Sections \ref{ss12} and \ref{ss14} can be computed using commands implemented by Claus Fieker in the MAGMA computer algebra system \cite{MR1484478}, so we can compute $\pi_\Q\left(H_\Q(\mathbf{S})\right)=\Bar H_\Q(\mathbf{S})$. The bottleneck of the computation is computing the class and unit groups of the number fields $K_h$, for $h\in\Ff_\Q$, which are needed for the construction of $A(q,\mathbf{S})$.

Also although Corollary \ref{finite} indicates that the algorithm computes $\sel$ in a finite amount of time, the size of $S_{up}$ is prohibitively large and in general we can only hope to get information using small primes. Nevertheless we can still put the algorithm in good use as in most cases bigger primes do not have a contribution in cutting down the set we already have.

\section{Examples}

In this section we give examples of how descent on superelliptic curves can be used to tackle some interesting number theoretical problems. Example \ref{exa2} is a preparatory example to demonstrate how the results from the local computations are obtained and combined to prove statements regarding the sets of rational points of the curve, or curves in question. In Example \ref{exa3} we show how descent can sometimes be the appropriate technique for solving generalized Fermat equations. In Example \ref{exa4} we consider a superelliptic curve, which although covers a plane cubic curve in an obvious way ($X\mapsto X^2,Z\mapsto Z^2$), the set of rational points of this curve is infinite, so it is impossible to construct the set of rational points of the superelliptic curve by pulling back rational points of the cubic curve. After using the algorithm described above we exclude all but one of the covers $D_\al$ due to local insolubility and we manage to compute all the rational points of this remaining cover $D_1$ by other means (found in \cite{MR2011330},\cite{MR2156713}). Finally, we compute $C(\Q)$ since $C(\Q)=\phi_1(D_1(\Q))$. The reasoning of this example, i.e. using covering techniques to transfer the problem to a different type of curves, is also used in the proof of Theorem \ref{theo1}, but unlike Example \ref{exa4}, all the curves involved can be defined over $\Q$. Theorem \ref{theo1} also demonstrates the significance of extending descent arguments to singular superelliptic curves.

\begin{exam}\label{exa2}
Consider the curve defined by the equation\[y^5=2x^5+x^4+2x^3+x^2+3x+3.\]This is ELS, but after applying the algorithm to this curve we obtain the results shown in Table ~\ref{tab:tab1}.\\ \begin{table}[h]\centering\begin{tabular}{| c |@{\ } c@{\ } |@{\ } c @{\ }|@{\ } c@{\ } | @{\ }c@{\ } c@{\ } c@{\ } |@{\ } c@{\ } c@{\ } c@{\ } | @{\ }c @{\ }|}\hline $p$ & 1 & 2 & 3 & 5 & \ldots & 17 & 19 & \ldots & 37 & 41\\ \hline $\displaystyle\# \bar H_\Q(\mathbf{S})\bigcap_{{l\primes} \atop{l\leq p}}r_l^{-1}(\mu_{\Q_l}(C(\Q_l)))$ & 25 & 25 & 25 & 2 & \ldots & 2 & 1 & \ldots & 1 & 0\\ \hline\end{tabular}\caption{Descent computation for $C:y^5=2x^5+x^4+2x^3+x^2+3x+3$.}\label{tab:tab1}\end{table}\\ \noindent Therefore\[\bar H_\Q(\mathbf{S})\cap r_5^{-1}\left(\mu_{\Q_5}(C(\Q_5))\right)\cap r_{19}^{-1}\left(\mu_{\Q_{19}}(C(\Q_{19}))\right)\cap r_{41}^{-1}\left(\mu_{\Q_{41}}(C(\Q_{41}))\right)=\emptyset\]which proves that $C(\Q)=\emptyset$.
\end{exam}

\begin{exam}\la{exa3}
In \cite{MR1915212}, Halberstadt and Kraus, consider the following four generalized Fermat equations
\begin{align}
16a^7+87b^7+625c^7=0\label{HK1}\\
11a^5+29b^5+81c^5=0 \label{HK2}\\
27a^5+16b^5+2209c^5=0\label{HK3}\\
32a^7+81b^7+187c^7=0\label{HK4}
\end{align}
These have solutions everywhere locally, but appear to have no rational points. The authors explain how the modular approach fails to show that the set of rational points is empty. We show how one can use descent to tackle all four of them.

Observe that the problem can be easily transferred to the one of finding rational points on superelliptic curves.
\begin{align*}
\eqref{HK1} \Leftrightarrow ( -b,2a,-c)\in C_1(\Q)\text{, where }C_1: Y^7=8(87X^7+625Z^7)\\
\eqref{HK2} \Leftrightarrow ( -a,3c,-b)\in C_2(\Q)\text{, where }C_2: Y^5=3(11X^5+29Z^5)\\
\eqref{HK3} \Leftrightarrow ( -a,2b,-c)\in C_3(\Q)\text{, where }C_3: Y^5=2(27X^5+2209Z^5)\\
\eqref{HK4} \Leftrightarrow ( -b,2a,-c)\in C_4(\Q)\text{, where }C_4: Y^7=4(81X^7+187Z^7)
\end{align*}  
Table ~\ref{tab:tab2} contains the results obtained when we perform descent on these curves. 
\begin{table}[h]\centering\begin{tabular}{| c | @{\ }c@{\ } | @{\ }c@{\ } | @{\ }c@{\ } | @{\ }c@{\ } | @{\ }c@{\ } c@{\ } c@{\ } | @{\ }c@{\ } |}\hline $p$ & 1 & 2 & 3 & 5 & 7 & \ldots & 23 & 29 \\ \hline $\displaystyle\# \bar H_\Q(\mathbf{S})\bigcap_{{l\primes} \atop{l\leq p}}r_l^{-1}(\mu_{\Q_l}(C_1(\Q_l)))$ & 49 & 0 &  &  &  &  &  &  \\ \hline $\displaystyle\# \bar H_\Q(\mathbf{S})\bigcap_{{l\primes} \atop{l\leq p}}r_l^{-1}(\mu_{\Q_l}(C_2(\Q_l)))$ & 0 &  &  &  &  &  &  &  \\ \hline $\displaystyle\# \bar H_\Q(\mathbf{S})\bigcap_{{l\primes} \atop{l\leq p}}r_l^{-1}(\mu_{\Q_l}(C_3(\Q_l)))$ & 5 & 5 & 5 & 1 & 1 & \ldots & 1 & 0 \\ \hline $\displaystyle\# \bar H_\Q(\mathbf{S})\bigcap_{{l\primes} \atop{l\leq p}}r_l^{-1}(\mu_{\Q_l}(C_4(\Q_l)))$ & 7 & 0 &  &  &  &  &  &  \\ \hline\end{tabular}\caption{Results for the generalized Fermat curves}\label{tab:tab2}\end{table}

\noindent Therefore $C_i(\Q)=\emptyset$ for $i=1,2,3,4$.

\end{exam} 

\begin{exam}\la{exa4}
Consider the curve $C$ in $\P^2(1,2,1)$, defined by 
\[(X,Y,Z)\in C\Leftrightarrow Y^3=(X^2-3Z^2)(X^4-2Z^4).\]This curve has three points at infinity $(1,1,0)$, $(1,\rho,0)$ and $(1,\rho^2,0)$, where $\rho$ is a primitive cube root of unity. So we know that it has at least one rational point. After applying the algorithm to this curve we obtain the results shown in Table ~\ref{tab:tab3}.
\begin{table}[h]\centering\begin{tabular}{| c |@{\ } c @{\ }|@{\ } c @{\ }| @{\ }c@{\ } |@{\ } c@{\ } | @{\ }c@{\ } | @{\ }c @{\ }|@{\ } c@{\ } |@{\ } c@{\ } c@{\ } | }\hline $p$ & 1 & 2 & 3 & 5 & 7 & 11 & 13 & 17 & \ldots\\ \hline $\displaystyle\# \bar H_\Q(\mathbf{S})\bigcap_{{l\primes} \atop{l\leq p}}r_l^{-1}(\mu_{\Q_l}(C(\Q_l)))$ & 243 & 243 & 9 & 3 & 3 & 3 & 3 & 1 & \ldots\\ \hline\end{tabular}\caption{Results for $C:y^3=(x^2-3)(x^4-2)$.}\label{tab:tab3}\end{table}

The element of $\bar H_\Q(\mathbf{S})$ remaining is the image of the point $(1,1,0)$ under $\mu_\Q$ which is equal to the identity element $1\Q^*A^{*3}$. This corresponds to a cover $\phi_1:D_1\rightarrow C$, where $D_1$ is a curve in $\P^5\times C$ defined as in \eqref{coverseq}, whose set of rational points is non-empty. We fix embeddings of $K_{(x^2-3)}$ and $K_{(x^4-2)}$ in $\bar\Q$ and index the six elements of $\Th_{\bar\Q}$ as $\vartheta_1=\sqrt{3},\vartheta_2=-\sqrt{3},\vartheta_3=\sqrt[4]{2},\vartheta_4=-\sqrt[4]{2},\vartheta_5=i\sqrt[4]{2}$ and $\vartheta_6=-i\sqrt[4]{2}$. $D_1$ is defined by the following relations
\begin{align*}
\left((u_1,\ldots, u_6),(X,Y,Z)\right)\in D_1  \Leftrightarrow & \exists\lambda\neq 0\text{ such that }\quad\lambda u_j^3=X-\vartheta_j Z\\ &\text{ for }1\leq j\leq 6\text{ and}\\
 & \lambda^2\prod_{j=1}^6u_j=Y
\end{align*}
This covers a curve $E'$ of genus $1$ in $\P^2$, defined over the number field $L:=\Q(\vartheta_3)$, given by the equation \[(U,V,W)\in E'\Leftrightarrow\quad V^3=(U-\vartheta_3W)(U^2+\vartheta_3^2W^2).\]We have the following commutative diagram\[\begin{CD}D_1 @>{\phi}>> C\\ @V{\kappa}VV @VVV \\ E' @>>{\tau}> \P^1\end{CD}\qquad\begin{diagram}\left((u_1,\ldots,u_6),(X,Y,Z)\right) &\rMapsto & (X,Y,Z) \\ \dMapsto & & \dMapsto \\ (Xu_3^2,(X-\vartheta_3Z)u_5u_6,Zu_3^2) & \rMapsto & (X,Z)\end{diagram}\]Note that $\kappa$ is a well defined rational map between two non-singular curves so it is actually a morphism. We can put $E'$ into Weierstrass form via a linear transformation, by moving the point $(\vartheta_3,0,1)$ to $\infty$. We obtain the Weierstrass model\[E:y^2z-8\vartheta_3^2yz^2=x^3-64z^3.\]The isomorphism of the two models is given by\[\psi:E\rightarrow E',\qquad \psi(x,y,z)=(\vartheta_3y,2\vartheta_3^3x,y-8\vartheta_3^2z).\]Using the package MAGMA we find that the Mordell-Weil rank of $E(L)$ is $1$. Since\[\tau\left(\kappa\left(\phi^{-1}(C(\Q))\right)\right)\subseteq\P^1(\Q),\] we are not interested in all the $L$-rational points of $E'$, only $(U,V,W)\in E'(L)$ such that $U/W\in \Q$. Determining these will give us $C(\Q)$. To solve this problem we can use ``Elliptic curve Chabauty'' \cite{MR2011330}. Fortunately this is implemented in MAGMA. Using the inbuilt MAGMA commands we find that\[\{(U,V,W)\in E'(L):U/W\in\Q\}=\{(1,1,0),(0,-\vartheta_3,1)\}.\]We deduce that $C(\Q)=\{(1,1,0)\}$.
\end{exam}

\begin{exam}
In \cite{MR2098395}, \cite{MR598878}, \cite{MR2332596} and  \cite{MR0078395} the authors consider a generalization of Lucas \lq\lq Square Pyramid\rq\rq\ problem, namely the determination of all pairs $(a,b)\in\Z_{>0}^2$ that satisfy the equation
\[
b^q=1^k+2^k+\ldots +a^k
\]
for some $q\geq 2$ and $k\geq 1$. Theorem \ref{theo1} below is already proved in \cite{MR2098395} where the authors determine all the solutions for all values of $q$ and for $1\leq k\leq 11$. It is suggested in \cite{bennett1} that this could have been extended to larger values of $k$. Nevertheless, we present here this special case since our proof, which involves descent, includes the determination of the full set of rational points (as opposed to the subset of integral points) of a singular superelliptic curve  of genus $7$. We prove the case with $q=3$ because current tools only allow the use of cubic curves for the intermediate steps (although recent developments, like the implementation of descent on $J_C$ in MAGMA, suggest that Chabauty's method for superelliptic curves should be available in the near future, which would open the possibility to solve cases with $q>3$). We consider the case with $k=9$ since this is a value of $k$ where the superelliptic curve involved in the computation is both singular and non-hyperelliptic.

\begin{theo}\label{theo1}
The only pair $(a,b)\in\Z_{>0}^2$ satisfying
\[
b^3=\sum_{i=1}^ai^9
\]
is $(1,1)$.
\end{theo}

\begin{proof}
After replacing the right hand side of the equation by the closed formula for the sum of the first $a$ ninth powers we get the equation 
\begin{equation}\label{bern}
b^3=\frac{1}{10}a^2(a + 1)^2(a^2 + a - 1)\left(a^4 + 2a^3 - \frac{1}{2}a^2 - \frac{3}{2}a + \frac{3}{2}\right)
\end{equation}
After the change of variables in the proof of Proposition \ref{changeofvars}, we see that the solutions $(a,b)$ correspond to rational points on the following genus $7$, singular, superelliptic curve $C$ in $\P^2(1,4,1)$
\[
C:\   Y^3=X^2(X+5Z)^2(X + 10Z)^2(X^2 + 30XZ + 100Z^2)(X^4 + 30X^3Z + 460X^2Z^2 + 2400XZ^3 + 4000Z^4)
\] 
\begin{equation}\label{map}
(a,b)\mapsto (10,10000b,(a-1))\in C(\Q)
\end{equation}
We have the descent map $\mu_\Q:C(\Q)\rightarrow A_\Q^*/\Q^*A_\Q^{*3}$ where the algebra $A_\Q$ is isomorphic to the product $\Q\times\Q\times\Q\times K_1\times K_2$ where $K_1=\Q[t]/(t^2 + 30t + 100)$ and $K_2=\Q[t]/(t^4 + 30t^3 + 460t^2 + 2400t + 4000)$. Denote the images of $t$ in $K_1$ and $K_2$ by $\th_1$ and $\th_2$ respectively. We can compute representatives in $A_\Q^*$ for the images of the five known rational points. These are shown in Table ~\ref{tab:tabknownpoints}.
\begin{table}[h]\centering\scalebox{1}{\begin{tabular}{|c|c|c|}
\hline 
$i$ & $P_i\in C(\Q)$ & $\mu_\Q(P_i)$ \\
\hline 
$1$ & $(1,1,0)$ & $\left[(1,1,1,1,1)\right]$ \\ 
\hline 
$2$ & $(0,0,1)$ & $\left[(1, 5, 10, \th_1, \th_2)\right]$ \\ 
\hline 
$3$ & $(-5,0,1)$ & $\left[(1, 75, 1, 1, 400(5\th_2^3 + 134\th_2^2 + 1860\th_2 - 5840))\right]$ \\ 
\hline 
$4$ & $(-10,0,1)$ & $\left[(2, 1, 300, 25(3\th_1 + 10), 20(\th_2^3 + 20\th_2^2 - 90\th_2 - 700))\right]$ \\ 
\hline 
$5$ & $(-10,1000,3)$ & $\left[(2, 1, 4, 25(3\th_1 + 10), 20(13\th_2^3 + 330\th_2^2 + 4380\th_2 + 9600))\right]$ \\  
\hline 
\end{tabular}}\caption{Representatives in $A_\Q^*$ of images of known points.}\label{tab:tabknownpoints}\end{table}

By Proposition \ref{genus}, the genus $G$ of the covers $D_\al$ is 
\[
G=q^{d-2}\left(\frac{d(q-1)}{2}-q\right)+1=3^7\left(\frac{9\times 2}{2}-3\right)+1=13123
\]
so the set $S_{up}$ (see Definition \ref{useful}) contains the rational primes which are less than $4\times 13123^2=688852516$. Performing the local computations for all of these primes would be unfeasible, but we do not need to, since after checking the primes $2$, $3$ and $5$ we exclude every element from $H_\Q(\mathbf{S})$ apart from the five we already know. Thus
\[
\sel=\mu_\Q\left(\left\{P_1,P_2,P_3,P_4,P_5\right\}\right).
\]
These elements correspond to five covers $\phi_i:D_i\rightarrow C$, where each $D_i$ is a curve in $\P^8\times C$, defined using a representative in $\mu_\Q(P_i)$ (see \eqref{coverseq} in Section \ref{imagemup}).

Now for each $1\leq i\leq 5$, we choose a combination of three factors from $\Ff_{\bar\Q}$ to form a cover $\kappa_i:D_i\rightarrow E_i$ where $E_i$ is a genus one curve. By choosing factors whose product is defined over $\Q$, we ensure that $E_i$ and $\kappa_i$ are defined over $\Q$. When choosing the three factors we also aim to obtain an elliptic curve $E_i$ which has finitely many rational points. As it turns out, for all $1\leq i\leq 5$, the subset $\{x, x+5,x+10\}\subseteq\Ff_{\bar\Q}$ satisfies the criteria we need, giving the five curves in Table ~\ref{tab:tab4}.
\begin{table}[h]\centering\begin{tabular}{| c | c | c | }
\hline 
$i$ & $(U,V,W)\in E_i\Leftrightarrow$ & $\kappa_i\left((u_h)_{h\in\Ff_{\bar\Q}},(X,Y,Z)\right)=$\\ 
\hline
$1$ & $V^3=U(U+5W)(U+10W)$ & $\left(Xu_{(x)}^2,Xu_{(x+5)}u_{(x+10)},Zu_{(x)}^2\right)$ \\ 
\hline
$2$ & $50V^3=U(U+5W)(U+10W)$ & $\left(Xu_{(x)}^2,Xu_{(x+5)}u_{(x+10)},Zu_{(x)}^2\right)$ \\ 
\hline
$3$ & $75V^3=U(U+5W)(U+10W)$ & $\left(Xu_{(x)}^2,Xu_{(x+5)}u_{(x+10)},Zu_{(x)}^2\right)$ \\
\hline
$4$ & $600V^3=U(U+5W)(U+10W)$ & $\left(Xu_{(x)}^2,Xu_{(x+5)}u_{(x+10)},Zu_{(x)}^2\right)$ \\
\hline
$5$ & $8V^3=U(U+5W)(U+10W)$ & $\left(Xu_{(x)}^2,Xu_{(x+5)}u_{(x+10)},Zu_{(x)}^2\right)$ \\
\hline
\end{tabular}\caption{The cubic curves $E_i$ and the maps $\kappa_i:D_i\rightarrow E_i$.}\label{tab:tab4}\end{table}\\

We note that there are isomorphisms $E_1\cong E_5$ and $E_3\cong E_4$ over $\Q$, but this is not relevant to the computation. We leave the coefficients of $V^3$ in $E_4$ and $E_5$ unchanged to remind the reader that they originate from multiplying the first three entries of the representatives of $\mu_\Q(P_i)$ shown in Table ~\ref{tab:tabknownpoints}. All five cubic curves are elliptic curves with Mordell-Weil rank equal to $0$ so we can determine the sets $E_i(\Q)$.
\begin{align*}
&E_1(\Q)=\left\{\ba{c} (0 , 0 , 1), (-10 , -10 , 3), (-20 , 10 , 3),\\ (-10 , 0 , 1), \fbox{(1 , 1 , 0)}, (-5 , 0 , 1) \ea\right\}\\
&E_2(\Q)=\left\{\fbox{(0 , 0 , 1)}, (-10 , 0 , 1), (-5 , 0 , 1)\right\}\\
&E_3(\Q)=\left\{(0 , 0 , 1), (-10 , 0 , 1), \fbox{(-5 , 0 , 1)}\right\}\\
&E_4(\Q)=\left\{(0 , 0 , 1), \fbox{(-10 , 0 , 1)}, (-5 , 0 , 1)\right\}\\
&E_5(\Q)=\left\{\ba{c} (0 , 0 , 1), (-10 , 0 , 1), (-5 , 0 , 1),\\ (-20 , 5 , 3), (2 , 1 , 0), \fbox{(-10 , -5 , 3)} \ea\right\}
\end{align*}
We then compute the pre-images of these sets under the maps $\kappa_i$. The box indicates that a point has one $\Q$-rational point in its $\kappa_i$-fiber. The rest of the points have no rational pre-image. As expected, $D_i(\Q)$ contains exactly one element for each $1\leq i\leq 5$. Using Corollary \ref{cor1} we get 
\[
C(\Q)=\bigcup_{i=1}^5\phi_i\left(D_i(\Q)\right)=\left\{P_1,P_2,P_3,P_4,P_5\right\}.
\]
These points correspond to all solutions $(a,b)\in \Q^2$ satisfying Equation \eqref{bern}. In particular, using the inverse of the map \eqref{map} we get that $P_1,P_2,P_3,P_4$ and $P_5$ correspond to $(1,1)$, ``$\infty$'', $(-1,0)$, $(0,0)$ and $(-2,1)$ respectively. Therefore $(1,1)$ is the only solution where both $a$ and $b$ are positive integers. 
\end{proof}
\end{exam}
\section*{Acknowledgements}
The author would like to thank Samir Siksek for his valuable comments and guidance, Dino Lorenzini for pointing out that descent would be applicable to the equations in Example \ref{exa3} and Brendan Creutz along with the rest of the Computational Algebra Group at the University of Sydney, for their excellent hospitality and the opportunity to include the algorithms described in this paper in MAGMA \cite{MR1484478}.

\bibliographystyle{amsplain}
\bibliography{superelliptic}

\end{document}